\documentclass[a4paper,11pt]{amsart}

\usepackage{a4wide}
\usepackage{tikz}
\usepackage[T1]{fontenc}
\usepackage{lmodern, comment}
\usepackage{mathtools}
\usepackage{amssymb}

\newcommand{\R}{\mathbb{R}}
\newcommand{\C}{\mathbb{C}}

\allowdisplaybreaks

\newif\ifdetails
\detailstrue
\newcommand{\DETAIL}[1]%
{\ifdetails\par\fbox{\begin{minipage}{0.9\linewidth}\textit{Detail:}
      #1\end{minipage}}\par\fi}
\newcommand{\TODO}[1]%
{\ifdetails\par\fbox{\begin{minipage}{0.9\linewidth}\textbf{TODO:}
      #1\end{minipage}}\par\fi}

\newtheorem{lemma}{Lemma}
\newtheorem{proposition}[lemma]{Proposition}
\newtheorem{theorem}[lemma]{Theorem}

\theoremstyle{remark}

\title[The distribution of subtrees in dense graphs]{The distribution of subtrees in dense graphs and the roots of the subtree polynomial}

\author{
Stephan Wagner
\and Ruoyu Wang}
\thanks{Department of Mathematics, Uppsala University, Uppsala, Sweden, E-mail: {\tt stephan.wagner@math.uu.se, ruoyu.wang@math.uu.se}; supported by the Swedish research council (Vetenskapsrådet), grant  2022-04030}

\begin{document}

\begin{abstract}
For a graph $G$ with $n$ vertices and a positive integer $k \leq n$, let $s_k(G)$ be the number of subtrees (subgraphs that are trees, not necessarily induced) of $G$ with $k$ vertices. The subtree polynomial of $G$ is $S(G;x) = \sum_{k=1}^n s_k(G) x^k$. In this paper, we consider dense connected graphs with a minimum degree that is linear in the number of vertices. We prove that the number of missing vertices in a random subtree is asymptotically Poisson-distributed and deduce that all the roots of the subtree polynomial have to be close to $0$.
\end{abstract}

\maketitle

\section{Introduction}

By a \emph{subtree} of a graph $G$, we mean any (not necessarily induced) subgraph that is a tree. For a graph $G$ with $n$ vertices and a positive integer $k \leq n$, we let $s_k(G)$ denote the number of subtrees of $G$ with $k$ vertices. The subtree polynomial of $G$ can now be defined in analogy to other graph polynomials (for example the independence polynomial or the domination polynomial) by
$$S(G;x) = \sum_{k=1}^n s_k(G) x^k.$$
It first appears in the graph-theoretical literature in Jamison's work on the mean subtree order \cite{jamison1983average,jamison1984monotonicity} and on alternating Whitney sums \cite{jamison1987alternating,jamison1990alternating} in the special case where $G$ itself is a tree. It occurs quite frequently as an auxiliary tool in the study of the mean subtree order, for example in the proof of the gluing lemma in~\cite{mol2019maximizing}.

In the recent paper of Chin, Gordon, MacPhee and Vincent \cite{chin2018subtrees}, the subtree polynomial was also considered for arbitrary graphs. Motivated by questions raised by Chin et al.~on the probability that a random subtree is spanning (which can be expressed as $\frac{s_n(G)}{S(G;1)}$ in our notation), the behaviour of the coefficients $s_k(G)$ was studied in \cite{wagner2021probability} for dense random graphs following the classical Erd\H{o}s-R\'enyi random graph model.

In particular, it was shown that the coefficients of the highest powers in the subtree polynomial, counting trees that are close to spanning, follow a Poisson limit law of the following type:

\begin{theorem}[{see \cite[Corollary 1]{wagner2021probability}}]\label{thm:random-graphs}
Consider the random graph $G = G(n,p)$, and suppose that $p \to p_{\infty} > 0$. As $n \to \infty$, we have, for every fixed nonnegative integer $k$,
$$\frac{s_{n-k}(G)}{s_n(G)} \overset{p}{\to} \frac{1}{k!} (ep_{\infty})^{-k}.$$
\end{theorem}

Related to this result, we remark that Luo, Xu and Tian \cite{luo2024random} proved unimodality of the subtree polynomial for various special classes of graphs. The distribution of the subtree sizes in a tree was shown to be typically (but not always) asymptotically normal, see \cite{ralaivaosaona2018distribution}.

The first aim of this paper is to show that a weaker assumption suffices to obtain a result similar to Theorem~\ref{thm:random-graphs}: an asymptotic Poisson law holds for arbitrary (non-random) graphs if the minimum degree is linear in the number of vertices. 

\begin{theorem}\label{thm:poisson}
Let $\alpha > 0$ be fixed, and consider a connected graph $G$ with $n$ vertices whose minimum degree $\delta$ is at least $\alpha n$. Set $\beta(G) = s_{n-1}(G)/s_n(G)$. We have, for $k = o(n^{1/3})$,
$$\frac{s_{n-k}(G)}{s_n(G)} = \frac{1}{k!} \beta(G)^{k} \big(1+O(kn^{-1/3})\big),$$
with an error term that is uniform in $k$ for $k \leq n^{1/3-\epsilon}$ if $\epsilon > 0$ is fixed.
\end{theorem}

Under the same assumption, it was already shown in~\cite{wagner2021probability} that the probability for a random subtree to be spanning is bounded below by a positive constant that only depends on $\alpha$. The key ingredient that we will use here to refine the argument, which is based on double counting, is a theorem of Pemantle and Peres \cite{pemantle2014concentration} on the concentration of certain functionals.

We can use Theorem~\ref{thm:poisson} to infer information on the roots of the polynomial $S(G;x)$ when $G$ is a dense graph. Heuristically, the above theorem implies that $S(G;x) \sim s_n(G) x^n e^{\beta(G)/x}$, and since the exponential function has no zeros, one expects the roots of $S(G;x)$ to be close to $0$. We make this argument precise and prove the following theorem, which is the second main result of this paper.

\begin{theorem}\label{thm:roots}
Let $\alpha > 0$ be fixed, and consider a connected graph $G$ with $n$ vertices whose minimum degree $\delta$ is at least $\alpha n$. Then we have, for every constant $C > 6$,
$$\max \{|x|\,:\, x \in \C,\, S(G;x) = 0 \} \leq \frac{C}{\alpha \log n}$$
if $n$ is sufficiently large.
\end{theorem}

It is interesting to compare this to the case where $G$ is a tree. Brown and Mol~\cite{brown2020roots} considered the roots of the subtree polynomial of a tree and showed in particular that the modulus of all roots is bounded above by the absolute constant $1 + \sqrt[3]{3}$. On the other hand, they conjectured that the roots of the subtree polynomial of a tree with $n$ vertices have to lie in the annulus $\{x \in \C\,:\, \frac12 \leq |x+\frac12| \leq \frac12 + \negthickspace \sqrt[n-1]{n-1}\}$. This would be in stark contrast to the roots of the subtree polynomial of a dense graph, which cluster around $0$ according to Theorem~\ref{thm:roots}.

\section{Poisson limit: proof of Theorem~\ref{thm:poisson}}

We employ the same double-counting strategy as in \cite{wagner2021probability}, but in a modified version, since the graphs we are dealing with are no longer close to regular. Let us start with some notation. We let $d(v)$ denote the degree of a vertex $v$ in $G$. For a tree $T$, we let $\ell(T)$ be the set of leaves of $T$. Finally, we write $S \leq T$ to indicate that $S$ is a subtree of $T$, and $T \vdash G$ to indicate that $T$ is a spanning tree of $G$. For a spanning tree $T$ of $G$, we define the weight
$$w(T) = \sum_{v \in \ell(T)} \frac{1}{d(v)}.$$
Let $v$ be a vertex and $S$ a spanning tree of $G - v$. The number of ways to extend $S$ to a spanning tree of $G$ (by attaching $v$ as a new leaf) is precisely $d(v)$. So we have
\begin{align}
s_{n-1}(G) &= \sum_{v \in V(G)} \sum_{S \vdash G - v} 1 \nonumber \\
&= \sum_{v \in V(G)} \sum_{S \vdash G - v} \frac{d(v)}{d(v)} \nonumber \\
&= \sum_{v \in V(G)} \sum_{S \vdash G - v} \frac{1}{d(v)} \sum_{T \vdash G, S \leq T} 1 \nonumber \\
&= \sum_{T \vdash G} \sum_{v \in \ell(T)} \frac{1}{d(v)} \nonumber \\
&= \sum_{T \vdash G} w(T) \label{eq:sn1}
\end{align}
by changing the order of summation.

Now we generalize this to $s_{n-k}$, where $k$ is sufficiently small. A tree $S \vdash G - \{v_1,v_2,\ldots,v_k\}$ can be extended in at least $\prod_{j=1}^k (d(v_j) - k)$ ways to a spanning tree $T$ of $G$ by adding an edge from each of the vertices $v_1,v_2,\ldots,v_k$ to one of its neighbours in 
$G - \{v_1,v_2,\ldots,v_k\}$. In this spanning tree $T$, $v_1,v_2,\ldots,v_k$ are all leaves. Conversely, there are clearly at most $\prod_{j=1}^k d(v_j)$ ways to extend $S$ to such a tree. Since
\begin{align*}
\prod_{j=1}^k (d(v_j) - k) &= 
\prod_{j=1}^k \Big( d(v_j) \Big( 1 - \frac{k}{d(v_j)} \Big) \Big)
\geq \prod_{j=1}^k d(v_j) \cdot \Big( 1 - \frac{k}{\delta} \Big)^k \\
&\geq \prod_{j=1}^k d(v_j) \cdot \Big( 1 - \frac{k^2}{\delta} \Big) \geq \prod_{j=1}^k d(v_j) \cdot \Big( 1 - \frac{k^2}{\alpha n} \Big),
\end{align*}
the number of ways to extend $S \vdash G - \{v_1,v_2,\ldots,v_k\}$ to a spanning tree in which $v_1,v_2,\ldots,v_k$ are leaves is
$$\prod_{j=1}^k d(v_j) \cdot \Big( 1 - O \Big( \frac{k^2}{n} \Big) \Big).$$
Now we can write (assuming only $k = o(\sqrt{n})$)
\begin{align*}
s_{n-k}(G) &= \sum_{\{v_1,v_2,\ldots,v_k\} \subseteq V(G)} \sum_{S \vdash G - \{v_1,v_2,\ldots,v_k\}} 1 \\
&= \sum_{\{v_1,v_2,\ldots,v_k\} \subseteq V(G)} \sum_{S \vdash G - \{v_1,v_2,\ldots,v_k\}} \prod_{j=1}^k \frac{d(v_j)}{d(v_j)} \\
&= \sum_{\{v_1,v_2,\ldots,v_k\} \subseteq V(G)} \prod_{j=1}^k \frac{1}{d(v_j)}
\sum_{S \vdash G - \{v_1,v_2,\ldots,v_k\}} \prod_{j=1}^k d(v_j) \\
&= \sum_{\{v_1,v_2,\ldots,v_k\} \subseteq V(G)} \prod_{j=1}^k \frac{1}{d(v_j)} \sum_{\substack{T \vdash G \\ \{v_1,v_2,\ldots,v_k\} \subseteq \ell(T)}} \Big( 1 + O \Big( \frac{k^2}{n} \Big) \Big) \\
&= \Big( 1 + O \Big( \frac{k^2}{n} \Big) \Big) \sum_{T \vdash G} \sum_{\{v_1,v_2,\ldots,v_k\} \subseteq \ell(T)} \prod_{j=1}^k \frac{1}{d(v_j)} .
\end{align*}

Our next task is to approximate the sum
$$\sum_{\{v_1,v_2,\ldots,v_k\} \subseteq \ell(T)} \prod_{j=1}^k \frac{1}{d(v_j)}$$
by
$$\frac1{k!} \Big( \sum_{v \in \ell(T)} \frac{1}{d(v)} \Big)^k = \frac{w(T)^k}{k!}.$$
Before we do so, let us make a few observations: first, note that $w(T)$ is bounded above by a positive constant under our conditions since
\begin{equation}\label{eq:1alpha_bound}
w(T) = \sum_{v \in \ell(T)} \frac{1}{d(v)} \leq \frac{n}{\delta} \leq \frac{1}{\alpha}.
\end{equation}
It follows that
\begin{equation}\label{eq:sum_upperbd}
    \sum_{\{v_1,v_2,\ldots,v_k\} \subseteq \ell(T)} \prod_{j=1}^k \frac{1}{d(v_j)} \leq \frac1{k!} \Big( \sum_{v \in \ell(T)} \frac{1}{d(v)} \Big)^k = \frac{w(T)^k}{k!} \leq \frac{\alpha^{-k}}{k!}.
\end{equation}
Next, we consider the average of $w(T)$ over all spanning trees. By~\eqref{eq:sn1}, it is
$$\frac{1}{s_n(G)} \sum_{T \vdash G} w(T) = \frac{s_{n-1}(G)}{s_n(G)} = \beta(G),$$
thus also $\beta(G) \leq \frac{1}{\alpha}$ by~\eqref{eq:1alpha_bound}. We can apply the following result due to Hladk\'y, Nachmias and Tran \cite{hladky2018local} (actually a slightly weaker statement than what they proved) to show that $\beta(G)$ is bounded below by a positive constant  as well:

\begin{lemma}[{see~\cite[Theorem 1.5]{hladky2018local}}]
For any $\epsilon,\alpha > 0$ there exists an integer $n_0$ such that for every graph with $n \geq n_0$ vertices and minimum degree at least $\alpha n$, the probability that a uniformly random spanning tree has less than $(e^{-1}-\epsilon)n$ leaves is at most $\epsilon$.
\end{lemma}

As a consequence, the expected number of leaves in a uniformly random spanning tree, i.e.,
$$\frac{1}{s_n(G)} \sum_{T \vdash G} |\ell(T)|,$$
is at least $(e^{-1} - o(1))n$, and since
$$w(T) = \sum_{v \in \ell(T)} \frac{1}{d(v)} \geq \frac{|\ell(T)|}{n}$$
for every spanning tree $T$, it follows that $\beta(G)$ is at least $e^{-1} - o(1)$. In particular, it is bounded below by a positive constant. Now we use a result of Pemantle and Peres to show that $w(T)$ is concentrated around its expected value $\beta(G)$.

\begin{lemma}[{see~\cite[Theorem 1.1]{pemantle2014concentration}}]\label{lem:pemantle_peres}
Let $G$ be a finite connected graph with vertex set $V$ and edge set $E$. Let $f: \{0,1\}^E \to \R$ be a 
function with Lipschitz constant $1$ with respect to the Hamming distance (i.e., if $\mathbf{x},\mathbf{x}' \in \{0,1\}^E$ only differ in one position, then $|f(\mathbf{x}) - f(\mathbf{x}')| \leq 1$). Encoding a spanning tree as an element of $\{0,1\}^E$ (with a $1$ standing for an edge that is part of the spanning tree and a $0$ for an edge that is not), $f$ also becomes a function on the set of spanning trees of $G$. Let $T$ denote a uniformly random spanning tree of $G$, and let $X = f(T)$. The random variable $X$ satisfies the concentration inequalities
$$P \big( X - \mathbb{E}(X) \geq a \big) \leq \exp \Big({-} \frac{a^2}{8|V|} \Big)$$
and
$$P \big( X - \mathbb{E}(X) \leq -a \big) \leq \exp \Big({-} \frac{a^2}{8|V|} \Big)$$
for every $a > 0$.
\end{lemma}

If we interpret an element of $\{0,1\}^E$ as a subgraph $H$ of $G$ (formed by the edges for which the corresponding entry is $1$), then we can define $w(H)$ in the same fashion as for spanning trees, namely by
$$w(H) = \sum_{v \in \ell(H)} \frac{1}{d(v)},$$
where $\ell(H)$ is the set of vertices whose degree in $H$ is $1$. With this, $w$ becomes a function with Lipschitz constant $\frac{2}{\delta}$ (adding or removing an edge only affects two vertices, so $w$ can only change by at most $\frac{2}{\delta}$). Thus, Lemma~\ref{lem:pemantle_peres} applies to $\frac{\delta}{2}w$, giving us
$$P \big( \big| w(T) - \mathbb{E}(w(T)) \big| \geq b \big) \leq 2 \exp \Big({-} \frac{ \delta^2 b^2}{32n} \Big) \leq 2 \exp \Big({-} \frac{\alpha^2 b^2}{32} n \Big),$$
where probability and expected value are taken with respect to the uniform measure on spanning trees. We apply this bound with $b = n^{-1/3}$ to find that the proportion of spanning trees for which $|w(T) - \beta(G)| = 
| w(T) - \mathbb{E}(w(T))| > n^{-1/3}$ is $O(e^{-\kappa n^{1/3}})$ for some constant $\kappa > 0$. In view of the upper bound~\eqref{eq:sum_upperbd}, we thus have
\begin{align*}   
s_{n-k}(G) &= \Big( 1 + O \Big( \frac{k^2}{n} \Big) \Big) \sum_{T \vdash G} \sum_{\{v_1,v_2,\ldots,v_k\} \subseteq \ell(T)} \prod_{j=1}^k \frac{1}{d(v_j)} \\
&= \Big( 1 + O \Big( \frac{k^2}{n} \Big) \Big) \sum_{\substack{T \vdash G \\
|w(T) - \beta(G)| \leq n^{-1/3}}
} \sum_{\{v_1,v_2,\ldots,v_k\} \subseteq \ell(T)} \prod_{j=1}^k \frac{1}{d(v_j)} + O \Big( s_n(G) \frac{\alpha^{-k}}{k!} e^{-\kappa n^{1/3}} \Big).
\end{align*}

Now assume that $|w(T) - \beta(G)| \leq n^{-1/3}$. In particular, $w(T)$ is bounded below by a constant $\eta > 0$ (provided $n$ is large enough). We consider the probability measure on $\ell(T)$ where every $v \in \ell(T)$ is selected with probability $\frac{1}{d(v)w(T)}$. Then
$$\frac{k!}{w(T)^k} \sum_{\{v_1,v_2,\ldots,v_k\} \subseteq \ell(T)} \prod_{j=1}^k \frac{1}{d(v_j)}$$
is precisely the probability that $k$ independently chosen random elements of $\ell(T)$ are distinct. The probabilities can be bounded as follows:
$$\frac{1}{d(v)w(T)} \leq \frac{1}{\delta w(T)} \leq \frac{1}{\delta \eta} \leq \frac{1}{\alpha \eta n}.
$$
Thus the probability of choosing $k$ distinct elements is at least
$$\prod_{j=0}^{k-1} \Big( 1 - \frac{j}{\alpha \eta n} \Big) \geq \Big( 1 - \frac{k}{\alpha \eta n} \Big)^k \geq 1 - \frac{k^2}{\alpha \eta n} = 1 - O \Big( \frac{k^2}{n} \Big).$$
So it follows that
\begin{align*}
    \sum_{\{v_1,v_2,\ldots,v_k\} \subseteq \ell(T)} \prod_{j=1}^k \frac{1}{d(v_j)}
&= \frac{w(T)^k}{k!} \Big( 1 - O \Big( \frac{k^2}{n} \Big) \Big) = \frac{1}{k!} \big(\beta(G) + O(n^{-1/3}) \big)^k \Big( 1 - O \Big( \frac{k^2}{n} \Big) \Big)\\
&= \frac{1}{k!} \beta(G)^k \Big( 1 + O \Big( \frac{k}{n^{1/3}} + \frac{k^2}{n} \Big) \Big)
\end{align*}
whenever $|w(T) - \beta(G)| \leq n^{-1/3}$. Here, it was used that $\beta(G)$ is bounded below by a positive constant to deduce $\big(\beta(G) + O(n^{-1/3}) \big)^k = \beta(G)^k (1 + O (kn^{-1/3}))$.
Putting everything together, we have
$$\frac{s_{n-k}(G)}{s_n(G)} = \frac{1}{k!} \beta(G)^k \Big( 1 + O \Big( \frac{k}{n^{1/3}} + \frac{k^2}{n} \Big) \Big) + O \Big( \frac{\alpha^{-k}}{k!} e^{-\kappa n^{1/3}} \Big).$$
Since we are assuming that $k = o(n^{1/3})$, we have $\frac{k^2}{n} = o(\frac{k}{n^{1/3}})$. Moreover, we have
$$\frac{\alpha^{-k}}{\beta(G)^k} e^{-\kappa n^{1/3}} \leq
\exp \Big( \lambda k - \kappa n^{1/3} \Big) = \exp \Big({-}\kappa n^{1/3} \Big(1 - \frac{\lambda k}{\kappa n^{1/3}} \Big) \Big)$$
for some constant $\lambda$, and this is also $o(\frac{k}{n^{1/3}})$.
So we can absorb all error terms into the $O(\frac{k}{n^{1/3}})$ (with a $O$-constant that only depends on $\alpha$ and $\epsilon$ if we have $k \leq n^{1/3-\epsilon}$):
$$\frac{s_{n-k}(G)}{s_n(G)} = \frac{1}{k!} \beta(G)^k \Big(1 + O \Big( \frac{k}{n^{1/3}} \Big) \Big).$$

\section{Roots of the subtree polynomial: proof of Theorem~\ref{thm:roots}}

We now apply the estimates from the previous section to the subtree polynomial. Set $F(y)=\sum_{k=0}^{n-1} \frac{s_{n-k}(G)}{s_{n}(G)} y^k$, so that $S(G;x)=s_n(G) x^n F(\frac{1}{x})$. To simplify notation, we drop the dependence on $G$ and write $s_i$ and $\beta$ instead of $s_i(G)$ and $\beta(G)$. Moreover, we set $r = \frac{\alpha \log n}{C}$. We prove that if $n$ is sufficiently large,
$$
    F(y) \neq 0 \text{ for } |y|\leq r
$$ 
by applying Rouché's Theorem to the functions $F(y)-e^{\beta y}$ and $e^{\beta y}$ (note that the latter has no zeros).

\begin{proposition}
    Under the conditions of Theorem~\ref{thm:roots}, we have, for sufficiently large $n$,
    \begin{equation}\label{eq:rouche}
        \left | F(y)-e^{\beta y} \right | < \left | e^{\beta y} \right |
    \end{equation}
    for all $y \in \C$ with $|y|=r = \frac{\alpha \log n}{C}$.
\end{proposition}

\begin{proof}
    First, as $\beta \leq \frac{1}{\alpha}$ by~\eqref{eq:1alpha_bound}, we have
    \begin{equation}\label{eq:rhs_bound}
        \left| e^{\beta y} \right|=e^{\beta \Re(y)}\geq e^{-\beta |y|}= e^{-\beta r} \geq e^{-r/\alpha} = n^{-1/C}.
    \end{equation}
    Next, we expand the left hand side of~\eqref{eq:rouche}. Set $K=\lfloor n^{1/6} \rfloor$. For $|y|=r$, we have
    \begin{align*}
\left | F(y)-e^{\beta y} \right | &= \left| \sum_{k=0}^{\infty} \Big( \frac{s_{n-k}}{s_n} - \frac{\beta^k}{k!} \Big) y^k \right| \\
        &\leq \left| \sum_{k \leq K} \left ( \frac{s_{n-k}}{s_n} - \frac{\beta^k}{k!} \right ) y^k \right| + \left | \sum_{k > K} \frac{\beta^k}{k!} y^k \right| + \left | \sum_{k > K} \frac{s_{n-k}}{s_n} y^k \right| \\
        &\leq \underbrace{\sum_{k \leq K} \left | \frac{s_{n-k}}{s_n} - \frac{\beta^k}{k!} \right | r^k}_{\Sigma_1} 
        + \underbrace{\sum_{k > K} \frac{\beta^k}{k!} r^k}_{\Sigma_2}
        + \underbrace{\sum_{k > K} \frac{s_{n-k}}{s_n} r^k}_{\Sigma_3}.
    \end{align*}
Here, $s_{n-k}$ is interpreted as $0$ if $k \geq n$. By our choice of $K$, Theorem \ref{thm:poisson} applies to the first sum $\Sigma_1$ with a uniform error term: there is a constant $A$ (that only depends on $\alpha$) such that, for $0 \leq k \leq K$,
    $$
        \left| \frac{s_{n-k}}{s_n} - \frac{\beta^k}{k!}  \right| \leq \frac{Ak}{n^{1/3}} \cdot \frac{1}{k!} \beta^k.
    $$
Thus,
\begin{equation*}
    \Sigma_1 \leq \sum_{k \leq K} \frac{Ak}{n^{1/3}} \cdot \frac{1}{k!} \beta^kr^k  \leq \frac{A}{n^{1/3}} \sum_{k \geq 1} \frac{1}{(k-1)!} \beta^kr^k = \frac{A \beta r}{n^{1/3}} e^{\beta r}.
\end{equation*}
Note that $\beta r \leq \frac{r}{\alpha} = \frac{\log n}{C}$. So we have
$$\Sigma_1 \leq \frac{A \log n}{C} n^{1/C-1/3}.$$
Next, we consider the sum $\Sigma_2$. Since $\beta r \leq \frac{\log n}{C}$, we have
$$\Sigma_2 \leq \sum_{k > K} \frac{(\log n)^k}{C^k k!}.$$
For large enough $n$, we have
$$\frac{\frac{(\log n)^{k+1}}{C^{k+1} (k+1)!}}{\frac{(\log n)^k}{C^k k!}} = \frac{\log n}{C(k+1)} \leq \frac{\log n}{Cn^{1/6}} \leq \frac12$$
for all $k \geq K = \lfloor n^{1/6} \rfloor$. Thus, we can bound $\Sigma_2$ by a geometric sum:
$$\Sigma_2 \leq \sum_{k > K} \frac{(\log n)^k}{C^k k!} \leq \frac{(\log n)^K}{C^K K!} \sum_{k > K} \frac{1}{2^{k-K}} = \frac{(\log n)^K}{C^K K!}.$$
By Stirling's formula, this goes faster to $0$ than any power of $n$.

Finally, let us consider $\Sigma_3$. Here, we set $L = \lfloor \frac{\alpha n}{2} \rfloor$ and 
split the range of $k$ further into $K < k \leq L$ and $k > L$:
$$\Sigma_3 = \sum_{K < k \leq L} \frac{s_{n-k}}{s_n} r^k + \sum_{k >  L} \frac{s_{n-k}}{s_n} r^k.$$
Let the sums be denoted by $\Sigma_{3,1}$ and $\Sigma_{3,2}$, respectively. For the first part, we use the bound
$$\frac{s_{n-k}}{s_n} \leq \frac{1}{\alpha^k k!} \left ( 1- \frac{k}{\alpha n} \right )^{-k},$$
(see \cite[Eq.~(1)]{wagner2021probability}). In the range $K < k \leq L$, this can be further bounded by
\begin{equation}\label{eq:quotient_midrange}
\frac{s_{n-k}}{s_n} \leq \frac{1}{\alpha^k k!} \left ( 1- \frac{k}{\alpha n} \right )^{-k} \leq \frac{2^k}{\alpha^k k!}.
\end{equation}
Now the same argument as for $\Sigma_2$ shows that
$$\Sigma_{3,1} \leq \sum_{k > K} \frac{2^kr^k}{\alpha^k k!} 
= \sum_{k > K} \frac{(2 \log n)^k}{C^k k!} \leq \frac{(2 \log n)^K}{C^K K!}$$
for sufficiently large $n$, and this also goes faster to $0$ than any power of $n$.

Finally, for the second part $\Sigma_{3,2}$, we use another estimate from \cite{wagner2021probability}, namely (see \cite[Eq.~(2)]{wagner2021probability})
$$s_1 + s_2 + \cdots + s_r \leq 2^r s_r.$$
This gives us
$$\Sigma_{3,2} = \sum_{k >  L} \frac{s_{n-k}}{s_n} r^k \leq \frac{r^n}{s_n} \sum_{k > L} s_{n-k} \leq \frac{r^n}{s_n} \cdot 2^{n-L} s_{n-L}.
$$
By~\eqref{eq:quotient_midrange}, it follows that
$$
\Sigma_{3,2} \leq \frac{2^n r^n}{\alpha^L L!},
$$
and again Stirling's formula shows that this goes to $0$ faster than any power of $n$.

    Putting everything together, we have $\Sigma_1 = O(n^{1/C-1/3}\log n)$ and $\Sigma_2,\Sigma_{3,1},\Sigma_{3,2} = O(n^{-a})$ for every positive constant $a$. Thus,
$$|F(y)-e^{\beta y}| \leq \Sigma_1 + \Sigma_2 + \Sigma_{3,1} + \Sigma_{3,2} = O(n^{1/C-1/3}\log n).$$
On the other hand, $|e^{\beta y}| \geq n^{-1/C}$ by~\eqref{eq:rhs_bound}. Since $\frac{1}{C} - \frac{1}{3} < -\frac{1}{C}$ by our choice of $C$, this proves~\eqref{eq:rouche} for sufficiently large $n$, and Rouch\'e's theorem yields the desired statement.
\end{proof}

\bibliographystyle{abbrv}
\bibliography{SubtreePolynomial}

\end{document}